\documentclass{amsart}
\usepackage[headings]{fullpage}
\usepackage{amsfonts,amssymb,amsthm,amsmath,euscript}
\usepackage{units}
\usepackage{subfigure}
\usepackage{graphicx}
\usepackage{tikz}
\usepackage{pinlabel}
\usepackage[colorlinks,citecolor=blue,linkcolor=red]{hyperref}
\input{xy}
\xyoption{all}

\newcommand{\Out}{\mathrm{Out}}

\newcommand{\Aut}{\mathrm{Aut}}

\newcommand{\norm}[1]{\left\Vert#1\right\Vert}

\newcommand{\R}{{\mathbb R}}

\newcommand{\Z}{{\mathbb Z}}

\newcommand{\A}{{\mathcal A}}

\newcommand{\Csec}{\EuScript{S}}

\newcommand{\BNS}{\Sigma}

\newcommand{\Hom}{\mathrm {Hom}}

\theoremstyle{plain}
\newtheorem{theorem}{Theorem}

\newtheorem{prop-defn}[theorem]{Proposition-Definition}

\newtheorem*{theorem:Lipschitzflows}{Theorem \ref{T:lipschtiz flows}} 
\newtheorem*{theorem:conesections}{Theorem \ref{T:cone_of_sections}}
\newtheorem*{theorem:splittings}{Theorem \ref{T:splittings}}
\newtheorem*{theorem:determinant}{Theorem \ref{T:determinant formula}}
\newtheorem*{theorem:continuity-convexity}{Theorem \ref{T:continuity/convexity again}}
\newtheorem*{theorem:cones-are-equal}{Theorem \ref{T:cones_are_equal}}
\newtheorem*{theorem:cone-is-BNScpt}{Theorem \ref{T:BNS}}
\newtheorem*{theorem:transversefoliations}{Theorem \ref{T:transverse foliations}}
\theoremstyle{definition}

\newtheorem*{remark*}{Remark}
\newtheorem*{remarks*}{Remarks}

\makeatletter
\let\c@equation\c@theorem
\makeatother
\numberwithin{equation}{section}

\begin{document}

\title[Unbounded asymmetry of stretch factors]{Unbounded asymmetry of stretch factors}
\author{Spencer Dowdall, Ilya Kapovich, and Christopher J. Leininger}

\address{\tt  Department of Mathematics, University of Illinois at Urbana-Champaign,
  1409 West Green Street, Urbana, IL 61801
  \newline \indent http://www.math.uiuc.edu/\~{}dowdall/, http://www.math.uiuc.edu/\~{}kapovich/, http://www.math.uiuc.edu/\~{}clein/} \email{\tt dowdall@illinois.edu,  kapovich@math.uiuc.edu, clein@math.uiuc.edu}


  
\begin{abstract}
A result of Handel--Mosher guarantees that the ratio of logarithms of stretch factors of any fully irreducible automorphism of the free group $F_N$ and its inverse is bounded by a constant $C_N$.  In this short note we show that this constant $C_N$ cannot be chosen independent of $N$.
\end{abstract}

\thanks{The first author was partially supported by the NSF
  postdoctoral fellowship, NSF MSPRF no. 1204814. The second author
  was partially supported by the NSF grant DMS-0904200 and by
  the Simons Foundation Collaboration grant no. 279836. The third
  author was partially supported by the NSF grant DMS-1207183. The third author acknowledges support from U.S. National Science Foundation grants DMS 1107452, 1107263, 1107367 ``GEAR Network".}

\subjclass[2010]{Primary 20F65, Secondary 57M, 37B, 37E}

\maketitle



Let $F_N$ be the free group of rank $N \geq 2$. An outer automorphism $\varphi\in \Out(F_N)$ is said to be \emph{fully irreducible} if no power of $\varphi$ preserves the conjugacy class of any proper free factor of $F_N$. In this case $\varphi$ has a well defined \emph{stretch factor} $\lambda(\varphi)$, which, for any non-$\varphi$--periodic conjugacy class $\alpha$ in $F_N$ and a free basis $X$ of $F_N$,  is given by
\[\lambda(\varphi) = \lim_{n\to \infty} \sqrt[n]{\norm{\varphi^n(\alpha)}_X},\]
where $\norm{\cdot}_X$ denotes cyclically reduced word length with respect to $X$.  As was observed in \cite{BH92} (see also \cite{HandelMosher-Parageom}), there exist fully irreducible elements $\varphi \in \Out(F_N)$ with the property that $\varphi$ and $\varphi^{-1}$ have {\em different} stretch factors:
\[ \lambda(\varphi) \neq \lambda(\varphi^{-1}).\]
However, the following result from \cite{HandelMosher-expansion} describes the extent to which they can differ.  To state their result precisely, let $N \geq 2$ and set
\[ C_N = \sup_{\varphi} \frac{\log(\lambda(\varphi))}{\log(\lambda(\varphi^{-1}))},\]
where $\varphi$ ranges over all fully irreducible elements of $\Out(F_N)$. 
\begin{theorem}[Handel-Mosher] \label{T:HM}
For $N\ge 2$, $C_N < \infty$.
\end{theorem}
An alternate proof of this result was more recently given by Algom-Kfir and Bestvina \cite{A-KB}.  While the proofs of this theorem appeal to the fact that $N$ is fixed, it is not clear that this dependence is necessary.  In this short note, we prove that in fact it is.
\begin{theorem}  
With $\{C_N\}_{N \geq 2}$ defined as above, $\displaystyle{\limsup_{N \to \infty} C_N = \infty}$.
\end{theorem}
\begin{proof}
The proof will appeal to a construction and analysis carried out in \cite{DKL1} and \cite{DKL2}.   To that end, let $F_3 = \langle a,b,c \rangle$ and consider the element $\varphi \in \Aut(F_3)$ defined by
\[ \varphi(a) = b \, , \quad \varphi(b) = b^{-1} a^{-1} b a c \, , \quad \varphi(c) = a.\]
It was shown in \cite[Example 5.5]{DKL1} that $\varphi$ is fully irreducible. Next, let
\[ G = F_3 \rtimes_\varphi \Z = \langle a,b,c,r \mid r^{-1} \, x \, r = \varphi(x) \mbox{ for all } g \in F_3 \rangle \]
be the free-by-cyclic group determined by $\varphi$, and let $u_0 \colon G \to \Z$ in $\Hom(G;\R) =  H^1(G;\R)$ be the associated homomorphism obtained by sending $r$ to $1$ and all other generators to $0$.  

In \cite{DKL1}, we construct a cone $\A \subset H^1(G;\R)$ containing $u_0$ with the property that every other primitive integral element $u \in \A$ has kernel $\ker(u)$ a finitely generated free group.
The action of $u(G) =\Z$ on $\ker(u)$ is generated by a \emph{monodromy automorphism} $\varphi_u \in \Aut(\ker(u))$ determining an expression of $G$ as a semidirect product $G \cong \ker(u) \rtimes_{\varphi_u} \Z$ with associated homomorphism $u$. One of the main results of \cite{DKL1} is that all such $\varphi_u$ are fully irreducible.

In \cite{DKL2}, we construct a strictly larger open, convex cone $\A \subsetneq \Csec \subset H^1(G;\R)$ and a function
\[ \mathfrak H \colon \Csec \to \R \]
that is convex, real analytic, and homogeneous of degree $-1$ (i.e., $\mathfrak H (t u) = \frac{1}{t} \mathfrak H(u)$) such that
\[ \log(\lambda(\varphi_u)) = \mathfrak H(u) \]
for any primitive integral class $u \in \A$. In fact this holds for all primitive integral $u \in \Csec$ with the appropriate interpretation of $\lambda(\varphi_u)$.  We also show that $\Csec$ is the cone on the component of the BNS-invariant $\BNS(G)$ \cite{BNS} containing $u_0$ \cite[Theorem I]{DKL2}  
and that $\A$ lies over the symmetrized BNS-invariant (that is, both $\A$ and $-\A$ project into $\BNS(G)$) \cite[Corollary 13.7]{DKL2}.
In fact, a key result of Bieri--Neumann-Strebel is that an integral class $u\in \Hom(G;\Z)$ has $\ker(u)$ finitely generated if and only if both $u$ and $-u$ lie in the $\BNS(G)$ \cite{BNS}.

The the homomorphism $-u_0$ has $\ker(-u_0) = \ker(u_0) = F_N$ and associated monodromy $\varphi^{-1}$, thus expressing $G$ as $F_N\rtimes_{\varphi^{-1}}\Z$. Since $\varphi^{-1}$ is also fully irreducible, the main result of \cite{DKL2} provides another open, convex cone $\Csec_-\subset H^1(G;\R)$ containing $-u_0$ and a corresponding convex, real analytic, homogeneous of degree $-1$ function $\mathfrak H_-\colon \Csec_-\to \R$. 
Since $-\A$ projects into $\BNS(G)$ and $\Csec_-$ is the cone on the component of $\BNS(G)$ containing $-u_0$, we see that $-\A\subset \Csec_-$. Thus $\mathfrak H_-$ calculates the inverse stretch factors
\[\mathfrak H_-(-u) = \log(\lambda(\varphi_u^{-1}))\]
for all primitive integral $u\in \A$.


Example 8.3 of \cite{DKL2} exhibits a primitive integral class $u_1 \in \Csec$ which lies on the boundary of $\A$ (see \cite[Figure 8]{DKL2}) 
for which $\ker(u_1)$ is {\em not} finitely generated. It follows that $-u_1$ is {\em not} in the BNS--invariant. The key observation is that $-u_1$ then necessarily lies on the boundary of $\Csec_-$ (since $-u_1\in \overline{-\A}\subset \overline{\Csec_-}$ but $-u_1\notin \Csec_-$). 
Let $\{u_n \}_{n=2}^\infty \subset \A$ be primitive integral classes protectively converging to $u_1$.  That is, there exists $\{t_n\}_{n=2}^\infty \subset \R$ so that $\displaystyle{ \lim_{n \to \infty} t_n u_n = u_1}$.  Since this convergence occurs inside $\Csec$, it follows that
\[ \lim_{n \to \infty} \mathfrak H(t_n u_n)  = \mathfrak H(u_1) < \infty.\]
On the other hand, since $\displaystyle{\lim_{n \to \infty} -t_n u_n = -u_1\in \partial \Csec_-}$, it follows from \cite[Theorem F]{DKL2} that
\[ \lim_{n \to \infty} \mathfrak H_-(-t_n u_n) = \infty.\]
Therefore, appealing to the homogeneity of $\mathfrak H$ and $\mathfrak H_-$, we have
\[ \lim_{n \to \infty} \frac{\log(\lambda(\varphi_{u_n}^{-1}))}{\log(\lambda(\varphi_{u_n}))}  = \lim_{n \to \infty} \frac{\mathfrak H_-(-u_n)}{\mathfrak H (u_n)} = \lim_{n \to \infty} \frac{\mathfrak H_-(-t_n u_n)}{\mathfrak H(t_n u_n)} = \infty. \qedhere\]

\end{proof}

\bibliography{mcapplication}{}
\bibliographystyle{alphanum}

\end{document}